\documentclass[10pt]{amsart}
\usepackage{graphicx}
\usepackage[mathscr]{eucal}
\usepackage{microtype}
\usepackage{amssymb}
\usepackage{mathrsfs}
\usepackage{xcolor}
\usepackage{color}
\usepackage[top=1in, bottom=1in, left=1in, right=1in]{geometry}
\usepackage{enumerate}
\usepackage{bm}
\usepackage[normalem]{ulem}
\usepackage[all]{xy}
\usepackage{tikz}

\usepackage{setspace} 

\DeclareMathOperator{\fs}{FS}

\newtheorem{claim}{Claim}

\theoremstyle{definition}

\author[D. Fern\'andez]{David J. Fern\'andez-Bret\'on}
\address{
Escuela Superior de F\'{\i}sica y Matem\'aticas\\
Instituto Polit\'ecnico Nacional\\
Av. Instituto Polit\'ecnico Nacional s/n Edificio 9, 
Col. San Pedro Zacatenco, Alcald\'{\i}a Gustavo A. Madero, 07738, CDMX, Mexico. 
}
\email{dfernandezb@ipn.mx}
\urladdr{https://dfernandezb.web.app}

\title[Folkman's theorem and the primes]{Folkman's theorem and the primes}



\begin{document}

\begin{abstract}
We provide two new proofs of the infinitude of prime numbers, using the additive Ramsey-theoretic result known as Folkman's theorem (alternatively, one can think of these proofs as using Hindman's theorem). This adds to the existing literature deriving the infinitude of primes from Ramsey-type theorems.
\end{abstract}

\maketitle

\section{Introduction}

A very recent and exciting vein of research is the search for new proofs in number theory utilizing Ramsey-theoretic results. Specifically regarding Euclid's proof that there are infinitely many primes, the first two instances were Alpoge's proof~\cite{alpoge} using van der Waerden's theorem~\cite{vanderwaerden}, as well as Granville's proof~\cite{granville} using the same fact together with Fermat's theorem on sums of squares within arithmetic progressions. More recently, Gasarch~\cite{gasarch} provided yet another proof using Schur's theorem~\cite{schur} together with Fermat's Last Theorem (for definiteness, one could say that Gasarch's proof uses the case $n=3$ of Fermat's Last Theorem). The same proof was discovered independently by Elsholtz~\cite{elsholtz}, who goes on to provide other proofs using results such as Roth's theorem (the case of length three of Szeméredi's theorem), or even Hindman's finite sums theorem~\cite{hindman-thm} together with an observation about the distances between consecutive $n$-th powers. More recent work along the same lines include G\"oral, Özcan and Sertba\c{s}' beautiful proof~\cite{goral-et-al} utilizing an extension of the polynomial van der Waerden's theorem of Bergelson and Leibman~\cite{bergelson-leibman-setpoly}, as well as Ad{\i}belli and G\"oral's proof~\cite{adibelli-goral} using Rado's theorem.

In this short paper, we insert ourselves within this tradition by using yet another Ramsey-theoretic result, Folkman's theorem, to prove that there are infinitely many primes. This result was named after Folkman by Graham, Rothschild and Spencer~\cite[\S 3.4]{graham-rothschild-spencer}, following a personal communication from Folkman to Graham and Rothschild; this result, however, was actually first published by Sanders~\cite{sanders}, and is also sometimes known as the Folkman--Rado--Sanders theorem because it follows from Rado's more general theorem~\cite[\S 3.3]{graham-rothschild-spencer} on partition regular equations. Folkman's theorem establishes that, for every colouring $c$ of the set of natural numbers $\mathbb N$ with finitely many colours, and for every $M\in\mathbb N$, there exists a set of (positive) natural numbers $X\subseteq\mathbb N$, with $|X|=M$, such that all sums obtainable from finitely many elements of $X$ (without repetitions) have the same colour, in other words, the set
\begin{equation*}
\fs(X)=\left\{a_1+\cdots+a_k\bigg|1\leq k\leq M\text{ and }a_1,\ldots,a_k\in X\text{ are distinct}\right\}
\end{equation*}
is $c$-monochromatic. So, Folkman's theorem is a generalization of Schur's theorem (Schur's being Folkman's particular case with $M=2$); alternatively, one can think of it as the finite version of Hindman's theorem (in Hindman's theorem, the corresponding set $X$ is infinite). After using Schur's and then Roth's theorem to prove the infinitude of prime numbers, Elsholtz~\cite[p. 254]{elsholtz} conjectured that it should also be possible to use Folkman's theorem (combined with some other number-theoretic results similar in flavour to Fermat's Last Theorem); immediately after he proceeds to provide another proof using Hindman's theorem (along with an observation about the differences between consecutive $n$-th powers).

This paper provides two different proofs using Folkman's theorem, so in a sense we confirm Elsholtz's conjecture; on the other hand, our proofs do not use any even moderately complicated number-theoretic result, but only combinatorics (the most advanced tool we use is the pigeonhole principle for the second of our proofs). One could say that our proofs sacrifice the simplicity of Schur's theorem, requiring to use the stronger Folkman's theorem, in exchange for being able to avoid any deep number theory. Of course, Folkman's theorem itself can be proved without using the fact that there are infinitely many primes. For example, one can use any of the original proofs (either Sanders'~\cite{sanders} or the one in~\cite[\S 3.4]{graham-rothschild-spencer}), or one can prove Hindman's theorem~\cite{hindman-thm} first and then deduce Folkman's using a compactness argument. Hindman's theorem itself can be proved either purely elementarily~\cite{baumgartner-short-proof-of-hindman}, or by means of ultrafilters~\cite[Corollary 5.10]{hindman-strauss} (even without compactness, it is possible to directly prove Folkman's theorem with ultrafilters as in~\cite[Theorem 20]{dfernandezb-monthly}); yet another alternative is to decide that rather than Folkman's theorem, we will use Hindman's theorem---even with this interpretation, our proofs are different from Elsholtz's from~\cite[Theorem 3]{elsholtz}. It is worth mentioning that, if one chooses to go the ultrafilter route, the proof of Hindman's theorem (or the proof of Folkman's theorem, for that matter) is of significantly less complexity than that of van der Waerden's theorem (cf. the proofs in e.g.~\cite[\S 14.1]{hindman-strauss} and~\cite[pp. 129-130]{dfernandezb-monthly}, or even the newest proof from~\cite{dinasso-new-vanderwaerden}) and so, from that perspective, our proofs are simpler than Alpoge's and Granville's (although definitely more complicated than Euclid's classical proof).

\section{The proofs}

We begin by establishing our notation and terminology. Given integers $a$ and $m>1$, the notation $a\mod m$ will stand for the unique number between $0$ and $m-1$ that is congruent to $a$ modulo $m$. Given two sequences (of integers) $s=(a_1,\ldots,a_i)$ and $t=(b_1,\ldots,b_j)$, we denote the concatenation of $s$ and $t$ with the symbol $s\frown t=(a_1,\ldots,a_i,b_1,\ldots,b_j)$. We will always use the letter $\mathbb P$ to denote the set of all prime numbers. Now, given a prime number $p\in\mathbb P$, recall that (by the fundamental theorem of arithmetic) for every nonzero integer $a$ there exist unique integers $\alpha\geq0$ and $A$ such that $a=p^\alpha A$, with $(p,A)=1$. We will use the notation $\nu_p(a)=\alpha$ and say that $\alpha$ is the $p$-adic order of $a$; we will also denote $\xi_p(a)=A$. Hence, if $\mathbb P$ is the set of all prime numbers, then we have
\begin{equation*}
n=\prod_{p\in\mathbb P}p^{\nu_p(n)}
\end{equation*}
and
\begin{equation*}
\xi_p(n)=\prod_{{q\in\mathbb P}\atop{q\neq p}}q^{\nu_q(n)}
\end{equation*}
for all $n\in\mathbb N$. An important (elementary) fact that we will use is that, for integers $a,b$ such that $\nu_p(a)<\nu_p(b)$, we have $\nu_p(a+b)=\nu_p(a)$.

\noindent {\it The first proof:}
Suppose that $\mathbb P$ is a finite set, say $|\mathbb P|=N$, and let $c$ be the colouring given by
\begin{equation*}
c(n)=(\nu_2(n)\mod 2,\xi_2(n)\mod 4)\frown(\xi_p(n)\mod p\big|p\in\mathbb P\setminus\{2\}).
\end{equation*}
By definition, $\xi_2(n)\mod 4$ is either 1 or 3, and $\xi_p(n)\mod p$ is between $1$ and $p-1$ for each $p\in\mathbb P\setminus\{2\}$, so that $c$ is a colouring with $4\prod_{p\in\mathbb P\setminus\{2\}}(p-1)$ colours. By Folkman's theorem, there exists a set $X$ with $M=N+1$ nonzero elements such that $\fs(X)$ is $c$-monochromatic.

\begin{claim}
For any two distinct $a,b\in X$, we have $\nu_2(a)\neq\nu_2(b)$.
\end{claim}

\begin{proof}[Proof of claim:]
Suppose, on the contrary, that $a,b\in X$ are distinct and $\nu_2(a)=\nu_2(b)=\alpha$. By monochromaticity of $\fs(X)$, the numbers $\xi_2(a)\mod 4,\xi_2(b)\mod 4$ are either both equal to $1$, or both equal to $3$; in any case it must be the case that $\xi_2(a)+\xi_2(b)\equiv 2\mod 4$. Since $a+b=p^\alpha(\xi_p(a)+\xi_p(b))$, this means that $\nu_2(a)=\alpha$ and $\nu_2(a+b)=\alpha+1$; since $\fs(X)$ is monochromatic (and the colour contains, in its first entry, the information about the parity of $\nu_2$), this is a contradiction.
\end{proof}

\begin{claim}
For each odd $p\in\mathbb P$ and for any two distinct $a,b\in X$, we have $\nu_p(a)\neq\nu_p(b)$.
\end{claim}

\begin{proof}[Proof of claim:]
Seeking for a contradiction, assume that $a,b\in X$ are distinct and $\nu_p(a)=\nu_p(b)=\alpha$. Since $\fs(X)$ is monochromatic, there is a $C$, $1\leq C\leq p-1$, such that $C\equiv\xi_p(a)\equiv\xi_p(b)\mod p$. Then $\xi_p(a)+\xi_p(b)\equiv 2C\mod p$. Since $(2,p)=1$ and $(C,p)=1$, we conclude $(2C,p)=1$; in particular, $(p,\xi_p(a)+\xi_p(b))=1$. However, since $a+b=p^\alpha(\xi_p(a)+\xi_p(b))$, the conclusion is that $\xi_p(a+b)=\xi_p(a)+\xi_p(b)\equiv 2C\not\equiv C\mod p$, contradicting the monochromaticity of the set $\fs(X)$.
\end{proof}

Therefore, distinct elements of $X$ always have distinct values for each of the $\nu_p$ functions. We choose at most $N$ elements of $X$ in the following way: let $a_1$ be the element of $X$ with least possible value for $\nu_2(a_1)$; then, choose $a_2\in X$ to be the one with least possible $\nu_3(a_2)$, unless this element is already $a_1$, in which case we do not choose $a_2$ yet. This process continues along the elements of $\mathbb P$: in general, once we are in the step corresponding to $p\in\mathbb P$, and assuming that we have already defined $a_1,\ldots,a_k$, and that for each $q\in\mathbb P$ with $q<p$ we have the element $a\in X$ with least possible $\nu_q(a)$ listed among the $a_i$, then we let $a_{k+1}$ be the element of $X$ with least possible $\nu_p$, unless it is already listed among the $a_i$ (in case it is already listed, we simply skip this step and choose $a_{k+1}$ in the step corresponding to the next element of $\mathbb P$). In the end, once we have obtained the full sequence $a_1,\ldots,a_t\in X$ (for some $t\leq N)$, since $|X|=N+1$ we may choose an $a\in X$ that is not listed among the $a_i$. Then, by construction, we have $\nu_p(a_1+\cdots+a_t)<\nu_p(a)$ for each $p\in\mathbb P$. Therefore $\nu_p(a_1+\cdots+a_t+a)=\nu_p(a_1+\cdots+a_t)$; since this happens for all prime numbers, we may conclude that $a_1+\cdots+a_t+a=a_1+\cdots+a_t$, hence $a=0$, a contradiction.
\hfill$\Box_{\text{First proof}}$

Our second proof uses slightly different ideas, being much more similar in spirit to e.g. Alpoge's proof that uses van der Waerden's theorem~\cite{alpoge}. Strictly speaking, the upcoming proof uses less colours, but it does require invoking a much bigger monochromatic set.

\noindent {\it The second proof:}
Suppose, once again, that the set $\mathbb P$ of prime numbers is finite, and let $|\mathbb P|=N$. We define a colouring of natural numbers $c$ by setting
\begin{equation*}
c(n)=(\nu_p(n)\mod 2\big|p\in\mathbb P).
\end{equation*}
Observe that this is a colouring with $2^N$ colours. Let
\begin{equation*}
M=(N+1)\prod_{p\in\mathbb P}p^4
\end{equation*}
and apply Folkman's theorem to obtain a set $X$ of $M$ many distinct (nonzero) numbers such that $\fs(X)$ is $c$-monochromatic.

\begin{claim}
For each $p\in\mathbb P$ and for each $\alpha$, there are no more than $p^4$ elements $a\in X$ such that $\nu_p(a)=\alpha$. 
\end{claim}

\begin{proof}[Proof of claim]
Suppose, seeking a contradiction, that there are more than $p^4$ such elements. Since $\xi_p(a)\mod p^2$ is one of $p^2$ possibilities, for all $a\in X$, by the pigeonhole principle one can find $p^2$ distinct elements $a_1,a_2,\ldots,a_{p^2}\in X$, along with some $0<A<p^2$, such that $\xi_p(a_i)\equiv A\mod p^2$ and $\nu_p(a_i)=\alpha$ for all $i\leq p^2$. Since $(A,p)=1$, there exists some $t$, $0<t<p^2$, such that $tA\equiv p\mod p^2$. Letting $b=a_1+\cdots+a_t$ and $B=\xi_p(a_1)+\cdots+\xi_p(a_t)$, we get
\begin{equation*}
b=a_1+\cdots+a_t=p^\alpha(\xi_p(a_1)+\cdots+\xi_p(a_t))=p^\alpha B,
\end{equation*}
where
\begin{equation*}
B=\xi_p(a_1)+\cdots\xi_p(a_t)\equiv A+\cdots+A=tA\equiv p\mod p^2,
\end{equation*}
meaning $\nu_p(B)=1$. Therefore $\nu_p(b)=\alpha+1$ while $\nu_p(a_1)=\alpha$, contradicting the monochromaticity of $\fs(X)$ (since $c$ includes the information about the parity of $\nu_p$).
\end{proof}
The attentive reader will note that the $p^4$ in the above claim is overkill; by being slightly more careful in the above proof, one can actually ensure that no more than $(p^2-p)(p^2-2)$ elements of $X$ have the same $\nu_p$ value.

Therefore, since $X$ has $M=(N+1)\prod_{p\in\mathbb P}p^4$ elements, one can successively thin out the set $X$, by going through each $p\in\mathbb P$ and removing, for each $\alpha$, all but one of the elements $a\in X$ with $\nu_p(a)=\alpha$. The previous claim ensures that, by doing this, we are keeping at least $\frac{1}{p^4}$ of the elements of our set. Therefore, at the end of the process, we are left with a subset $Z\subseteq X$, with $|Z|\geq\left(\prod_{p\in\mathbb P}\frac{1}{p^4}\right)|X|=N+1$, such that for each $p$ and any two distinct $a,b\in Z$, we must have $\nu_p(a)\neq\nu_p(b)$. We now work on the elements of $Z$, in exactly the same way as in our first proof, in order to obtain a list $a_1,\ldots,a_t$ (for some $t\leq N)$ of elements of $Z$ such that, for all $p\in\mathbb P$, the element $a\in Z$ with least possible value for $\nu_p(a)$ is already listed among the $a_i$. Just as in our first proof, since $|Z|>N$, we may choose an $a\in X$ not listed among the $a_i$, so that $\nu_p(a_1+\cdots+a_t+a)=\nu_p(a_1+\cdots+a_t)$ for all $p\in\mathbb P$ and, since the $p\in\mathbb P$ are all the prime numbers, we may conclude that $a_1+\cdots+a_t+a=a_1+\cdots+a_t$ so that $a=0$, a contradiction.

\hfill$\Box_{\text{Second proof}}$

\section*{Acknowledgements}

The author's interest in the applications of Ramsey theory to number theory was first sparked at the 2024 RaTLoCC (Ramsey Theory in Logic, Combinatorics and Complexity) conference in Pisa (especially after discussions with Haydar G\"oral), and later on it was revived at the 2025 workshop on Infinitary Proof Theory at CMO-BIRS; the author is therefore grateful to the organizers of these events, as well as to the two anonymous referees for a thorough reading and useful suggestions on the paper. Finally, the author was partially supported by Secihti's grant CBF2023-2024-334, as well as by IPN's internal grants SIP-20253559 and SIP-20260817.

\end{document}